\documentclass{tsp-short}
\usepackage{amsfonts}
\usepackage{amssymb}
\usepackage{newlfont}
\usepackage{amsthm}
\usepackage{euscript}
\usepackage{amsmath}
\usepackage{graphicx}
\usepackage{dsfont}
\usepackage{enumerate}
\usepackage{amsmath}
\usepackage{color}

\newcommand{\llim}{\mathop \mathrm{l.i.m.}}

\newcommand{\esssup}{\mathop \mathrm{ess\ sup}}

\newtheorem{thm}{Theorem}[section]
\newtheorem{lem}{Lemma}[section]

\newtheorem{prp}{Proposition}[section]
\theoremstyle{remark}
\newtheorem{rem}{Remark}[section]
\theoremstyle{definition}
\newtheorem{dfn}{Definition}[section]

\theoremstyle{plain}

\newcommand{\cF}{{\cal F}}
\newcommand{\vp}{\varphi}
\newcommand{\vf}{\varphi}

\newcommand{\R}{\mathds{R}^d}
\renewcommand{\P}{\mathds{P}}
\newcommand{\E}{\mathds{E}}

\begin{document}

\title[On exponential decay of a distance]{On exponential decay of a distance between solutions of an SDE with non-regular drift}

\author{O. Aryasova}
\address{Institute of Geophysics, National Academy of Sciences of Ukraine,
32 Palladin ave., 03142, Kyiv, Ukraine; National Technical University of Ukraine ``Igor Sikorsky Kyiv Politechnic Institute'', Kyiv, Ukraine}
\email{oaryasova@gmail.com}

\author{A. Pilipenko}
\address{Institute of Mathematics,  National Academy of Sciences of
Ukraine, 3 Tereshchenkivska str., 01601, Kyiv, Ukraine; National Technical University of Ukraine ``Igor Sikorsky Kyiv Politechnic Institute'', Kyiv, Ukraine}
\email{pilipenko.ay@gmail.com}
\thanks{The research is partially supported by the Alexander von Humboldt Foundation (Project ``Singular diffusions: analytic and stochastic approaches'')}

\dedicatory{To the memory of our colleague Sergey Makhno}


\begin{abstract}
We consider a multidimensional stochastic differential equation
with a Gaussian noise and a drift vector having a jump
discontinuity along a hyperplane. The large time behavior of the
distance between two solutions starting from different points is
studied. We find a sufficient condition for the exponential decay
of the distance if the drift does not satisfy a dissipative
condition on a given hyperplane.
\end{abstract}

\maketitle \thispagestyle{empty}
\section{Introduction}
Consider a $d$-dimensional stochastic differential equation (SDE)
\begin{equation}\label{eq_main}
\left\{
\begin{aligned}
d\vp_t(x)&=\left(-\lambda\vp_t(x)+\alpha(\vp_t(x))\right)dt+\sum_{k=1}^{m}\sigma_{k}(\vp_t(x))dw_{k}(t), t\geq0,\\
\vp_0(x)&=x,\\
\end{aligned}\right.
\end{equation}
where $x=(x^1,\dots, x^d)\in\mathbb{R}^d$, $\lambda>0$, $(w(t))_{t\geq0}=(w_1(t),\dots,w_m(t))_{t\geq0}$ is a standard $m$-dimensional Wiener process, $\alpha:\R\to\R$ and $\sigma=(\sigma_1,\dots,\sigma_m):\R\to\R\times\mathds{R}^{m}$ are measurable functions.

It is well known that if $\alpha, \sigma_k$ are Lipschitz continuous and  $\lambda $ is large enough,
then the distance between solutions $\vp_t(x_1)$ and $\vp_t(x_2)$ to \eqref{eq_main}, which starting from two different points $x_1$ and $x_2$,
converges to 0 in $L_p(\Omega,\cF,\P)$  as $t\to\infty$,  (e.g., \cite{Ito+64}, \cite{DaPrato+96}). Moreover, the solutions converge themselves to a stationary solution of \eqref{eq_main}.
Lipschitz continuity of $\alpha$ may be relaxed; it can be replaced, for example, by dissipative assumption. It worth noting a recent work by Flandoli et al. \cite{Flandoli+2017}. The authors consider an SDE with drift belonging to $C^1$ and being dissipative out of some bounded set $U$ and  a diffusion coefficient $\sigma$ being constant. They prove the synchronization of the flow for large enough $\sigma$. Despite the drift is not supposed to be globally dissipative, the technique of their paper does not allow to consider a discontinuous drift.

In one-dimensional case, results on exponential decay of a distance between solutions to an SDE with non-regular drift were obtained in \cite{ Aryasova+2020, Aryasova+12}.

We discuss a similar problem if $\sigma_k$ are Lipschitz continuous but $\alpha$ may have a jump discontinuity at a
hyperplane. We do not assume that coefficients of the equation satisfy dissipative
conditions, and the results on convergence of solutions  to zero are new.  As a corollary of our results
on convergence of distance between solutions, we get existence and uniqueness of
 a stationary solution of
$$
\vp_t=\vp_s+\int_s^t(-\lambda\vp_u+\alpha(\vp_u))du+\sum_{k=1}^m\int_s^t\sigma_k(\vp_u)dw_k(u), \ s\leq t.
$$

Note that all our results concern strong solutions to SDEs, i.e., all solutions are defined
on the given probability space
and expectations are taken with respect to the given probability measure.
If one is interested in weak solutions and a distance between distributions
of $\vf_t(x)$ and $\vf_t(y)$, then
assumptions on coefficients may be relaxed essentially, see for example
   \cite{Khasminskii12, Kulik17}.




\section{Large time behavior of the distance between two solutions}

We consider the SDE \eqref{eq_main}.
Denote
$$
S=\{x\in\R: x^d=0\},
$$
$$
\mathds{R}^d_+=\{x\in\R: x^d>0\},
\mathds{R}^d_-=\{x\in\R: x^d<0\}.
$$

In what follows we assume that coefficients of \eqref{eq_main} satisfy the following conditions.
 \begin{enumerate}
\item [(A1)] The function $\alpha$ is bounded.
\item [(A2)]{\it Lipschitz continuity on $\mathbb{R}^d_{\pm}$}: There exists $\tilde{K}_{\alpha}>0$ such that for all $x,y\in \mathbb{R}^d_+$ or $x,y\in \mathbb{R}^d_-$,
    $$
    |\alpha(x)-\alpha(y)|\leq \tilde{K}_{\alpha}|x-y|.
    $$
\end{enumerate}
It follows from (A2) that for all $\tilde{x}\in S$, there exist limits
$$
\alpha_+(\tilde{x}):=\lim_{\substack{x\to \tilde{x},\\ x\in \mathbb{R}_+^d}}\alpha(x), \alpha_-(\tilde{x}):=\lim_{\substack{x\to \tilde{x}, \\ x\in \mathbb{R}_-^d}}\alpha(x).
$$
\begin{enumerate}
\item [(B1)] The function $\sigma$ is bounded.
\item [(B2)]{\it Lipschitz continuity on $\R$}: There exists $\tilde{K}_{\sigma}>0$ such that for all $x,y\in\R$,
    $$
        |\sigma(x)-\sigma(y)|\leq \tilde{K}_{\sigma}|x-y|.
    $$
   \item [(B3)]{\it Uniform ellipticity}: There exists a constant $B_\sigma>0$ such that for all $x\in \R$, $\theta\in\R$,

\begin{equation}\label{eq_ellipticity}
\theta^\ast\sigma(x)\sigma^\ast(x)\theta\geq B_\sigma|\theta|^2.
\end{equation}
\end{enumerate}

Under these assumptions there exists a unique strong solution to \eqref{eq_main} (see, for example, \cite{Veretennikov81}).

\begin{rem}\label{Remark_zero_time}
Note that since $\sigma$ is uniformly elliptic, the solution to equation \eqref{eq_main} spends zero time on $S$. So we can redefine the function $\alpha$ on $S$ in an arbitrary way.
\end{rem}

The main result of the paper is following:
\begin{thm}\label{Theorem_convergence}
Let  conditions (A1), (A2), (B1), (B2), (B3) hold. Then for any $p\geq 1:$
$$ \exists\Lambda=\Lambda(\alpha,\sigma)>0 \ \forall\lambda>\Lambda \ \exists C_1=C_1(\lambda,\alpha,\sigma)>0 \ \exists C_2=C_2(\lambda,\alpha,\sigma)>0: \ \forall x,y\in\R,
$$
\begin{equation}\label{eq_inequality_difference}
  (\E|\vp_t(y)-\vp_t(x)|^p)^{\frac{1}{p}}\leq C_1 e^{-C_2t}|y-x|.
\end{equation}
  Here $\vp_t(x)$ is a solution to equation \eqref{eq_main} starting at the point $x$.
\end{thm}
\begin{rem}
The  values of $\Lambda$, $C_1$, $C_2$ will be defined in the proof.
\end{rem}
\begin{proof}
It can be checked (see \cite{Aryasova+17}, Theorem 4 and (51)) that $\vp_t(\cdot)$ is G\^{a}teaux differentiable in $L_p(\Omega,\cF,\P)$, $p>0$, i.e.,
for all $x\in\mathds{R}^d$ there exists $\nabla\vp_t(x)$ such that for all $v\in\mathds{R}^d$, $t\geq0$,
\begin{equation}\label{eq_derivative_main2}
\mathds{E}\left|\frac{\varphi_t(x+\varepsilon v)-\varphi_t(x)}{\varepsilon}-\nabla\vp_t(x)v\right|^p\to 0, \ \varepsilon\to 0.
\end{equation}
Moreover, the process $(\nabla\vp_t(x))_{t\geq0}$ is continuous in $x$ in $L_p(\Omega,\cF,\P)$, $p>0$. This imply that
for $x,y \in\R$,
\begin{equation}\label{eq_Newton_Leibnitz}
    \vp_t(y)-\vp_t(x)=\int_0^1\big(\nabla\vp_t(x+\xi(y-x)),y-x\big)d\xi,
\end{equation}
where $\nabla\vp_t(\cdot)$ is the 
 derivative.

Using \eqref{eq_Newton_Leibnitz} we get
\begin{equation}\label{eq_ineq_difference_of_solutions}
\mathds{E}|\vp_t(y)-\vp_t(x)|^p\leq \mathds{E}(\int_0^1 \big|(\nabla\vp_t(x+\xi(y-x)),y-x)\big|d\xi)^p
\leq |y-x|^p\cdot\sup_{z\in\R}\mathds{E}|\nabla\vp_t(z)|^p.
\end{equation}
So to obtain \eqref{eq_inequality_difference} we need to get an estimate for
$\sup_{z\in\R}\mathds{E}|\nabla\vp_t(z)|^p$. We consider the case $p=1$
only. The general case can be considered similarly.

If $\alpha_+(x)=\alpha_-(x), x\in S$,  then $Y_t(x):=\nabla\vp_t(x)$ is a solution to the SDE
\begin{equation}\label{eq_smooth_coef}
\left\{
\begin{aligned}
  dY_t(x)&=[-\lambda+\nabla\alpha(\vp_t(x))]Y_t(x)dt+\sum_{k=1}^{m}\nabla\sigma_{k}(\varphi_t(x))Y_t(x)dw_{k}(t), t\geq0,\\
Y_0(x)&=E,\\
\end{aligned}\right.
 \end{equation}
where $E$ is a $d\times d$-identity matrix.
  This formula is well known when $\alpha, \sigma\in C^1(\R)$. For Lipschitz continuous functions $\alpha$ and $\sigma$ the result can be found in \cite{Bouleau+2010}, Th. 3.3.1.
 \begin{rem} It follows from Rademacher's theorem that the Lipschitz continuous functions $\alpha$ and $\sigma$ are differentiable almost everywhere with respect to the Lebesgue measure. We define $\nabla\alpha(\vp_t(x))$, $\nabla\sigma(\vp_t(x))$  in an arbitrary way at the points where they do not exist. Since $\sigma$ is non-degenerate, the distribution of $\vp_t(x)$ is absolutely continuous. So $\nabla\alpha(\vp_t(x))$, $\nabla\sigma(\vp_t(x))$ are defined uniquely up to the set of probability zero.
\end{rem}

If $\alpha_+(x)\neq\alpha_-(x), x\in S,$ then the distributional derivative of $\alpha$ is equal to
  $$
  \nabla\alpha(x)+D(x)\delta_S, x\in\R.
  $$
  Here $\delta_S$ is the standard surface measure on $S$ (if $d=1$, $\delta_S(x)$ is the Dirac delta function),
  and
$$
D(x)=\begin{pmatrix}
    0 & \cdots & 0 &\alpha_{+}^1(x)-\alpha_-^1(x)  \\
    \vdots & \ddots & \vdots & \vdots \\
    0 & \cdots & 0 & \alpha_{+}^d(x)-\alpha_-^d(x)
  \end{pmatrix}, \ x\in S.
$$

Formally, in this case the integral form of equation \eqref{eq_smooth_coef} becomes
\begin{multline}\label{eq_with_delta_S}
 Y_t(x)=E+\int_0^t[-\lambda+\nabla\alpha(\vp_s(x))]Y_s(x)ds+\\
 \int_0^tD(\vp_s(x))Y_s(x)\delta_S(\vp_s(x))ds+\sum_{k=1}^{m}\int_0^t\nabla\sigma_{k}(\varphi_s(x))Y_s(x)dw_{k}(s).
\end{multline}

It was proved in \cite{Aryasova+17} that $Y_t(x)$ is a solution to equation \eqref{eq_with_delta_S}, where by
$$
\int_0^tD(\vp_s(x))Y_s(x)\delta_S(\vp_s(x))ds
$$
we mean the integral with respect to the local time of the process $(\vp_t(x))_{t\geq 0}$ on the hyperplane $S$:
$$
\int_0^tD(\vp_s(x))Y_s(x)dL_s^S(\vp(x)),
$$
where
\begin{equation*} \label{eq_local_time_S}
L_t^S(\vp(x)):=\llim_{\varepsilon\downarrow 0}\frac{1}{2\varepsilon}\int_0^t\mathds{1}_{|\langle\vp_{s}(x),e_d\rangle|\leq\varepsilon }ds, \ e_d=(0,0,\dots,0,1).
\end{equation*}
Note that the local time of the process $(\vp_t(x))_{t\geq 0}$ on the hyperplane $S$ coincides with the local time of the $d$-th coordinate of the process $(\vp_t(x))_{t\geq0}$ at the point $0$, which is defined by the formula
\begin{equation}\label{eq_local_time_0}
L_t^0(\vp^d(x))=\llim_{\varepsilon\downarrow 0}\frac{1}{2\varepsilon}\int_0^t \mathds{1}_{|\varphi_s^d(x)|\leq\varepsilon}ds.
\end{equation}


Then equation \eqref{eq_with_delta_S} can be rewritten as follows
\begin{multline}\label{eq_derivative_1}
 Y_t(x)=E+\int_0^t[-\lambda+\nabla\alpha(\vp_s(x))]Y_s(x)ds+\\
 \int_0^tD(\vp_s(x))Y_s(x)dL_s^0(\vp^d(x))+\sum_{k=1}^{m}\int_0^t\nabla\sigma_{k}(\varphi_s(x))Y_s(x)dw_{k}(s).
\end{multline}
It is known that there exists a unique strong solution to equation \eqref{eq_derivative_1} (see, for example, \cite{Protter04}, Ch. V, Th. 7).

Set
\begin{equation}\label{eq_K_alpha_def}
K_{\alpha}:=\esssup_{x\in\R}|\nabla\alpha(x)|,
\end{equation}
\begin{equation}\label{eq_K_sigma_def}
K_{\sigma}:=\esssup_{x\in\R}|\nabla\sigma(x)|.
\end{equation}
Here and below we denote by $|\cdot|$ both the Euclidean norm of vectors and the Hilbert-Schmidt norm of matrices. Note that (A2), (B2) are satisfied with $\tilde{K}_\alpha=K_{\alpha}$, $\tilde{K}_\sigma=K_{\sigma}$, respectively. Put
$$
\|D\|_\infty=\sup_{x\in S}|D(x)|.
$$
Define
$$
h(t)=(2\lambda-2K_{\alpha}-K_{\sigma}^2)t-2\|D\|_\infty L_t^0(\vp^d(x)).
$$
\begin{lem}\label{Lemma_moment_of_Y} For all $T>0$,
$$
\sup_{t\in[0,T]}\mathds{E}e^{h(t)}|Y_t(x)|^2\leq d.
$$
\end{lem}
The proof of Lemma follows from It\^{o}'s formula. For details, see Appendix.
\vskip 20 pt

Using the H\"older inequality we obtain
\begin{multline}\label{eq_Y_moment}
\mathds{E}|Y_t(x)|\leq \left(\mathds{E} e^{h(t)}|Y_t(x)|^2\right)^{1/2}\left(\mathds{E}e^{-h(t)}\right)^{1/2}\leq \\
d^{1/2} e^{\left(-\lambda+K_{\alpha}+\frac12K_{\sigma}^2\right)t}\left(\mathds{E}e^{2\|D\|_\infty L_t^0(\vp^d(x))}\right)^{1/2}.
\end{multline}

\begin{lem}\label{Lemma_estimate_local_time} For each $t>0$,
\begin{equation}\label{eq_times_inequality_2}
\sup_{x\in\R}\mathds{E}L_t^0(\vp^d(x))\leq \frac{\rho(t,\lambda)}{B_\sigma},
\end{equation}
where
$$
\rho(t,\lambda)=\|\alpha^d\|_{\infty}t+\left(1+\frac{2}{3}\lambda t\right)\sqrt{\left(\frac{\|\alpha^d\|_{\infty}^2}{2\lambda}+\|\sigma\|_{\infty}^2\right)t}
$$
and $\|\alpha\|_\infty=\sup_{x\in\mathbb{R}^d}|\alpha(x)|$, $\|\sigma\|_{\infty}=\sup_{x\in\R}|\sigma(x)|$, $B_\sigma$ is the uniform ellipticity constant from equation \eqref{eq_ellipticity}.
\end{lem}
To prove the Lemma we use Tanaka's formula. See Appendix for details.
\vskip 20 pt

It is well known that  $L_t^S(\vp(x))$ is a W-functional of the Markov process $(\vp_t(x))_{t\geq0}$ (see \cite{Dynkin63}, Ch. 6--8 for theory and terminology). Then the following estimates on the moments of $L_t^S(\vp(x))$ are true.
\begin{prp}[\cite{Gikhman+04_II}, Ch. II, \S 6, Lemma 3]\label{Proposition_Gikhman} For all $n\geq1$, $t>0$,
$$
\sup_{x\in\R}\mathds{E}\left(L_t^S(\vp(x))\right)^n\leq n!\left(\sup_{x\in\R}\mathds{E}L_t^S(\vp(x))\right)^n.
$$
\end{prp}
Since $L_t^0(\vp^d(x))=L_t^S(\vp(x))$, then using Proposition \ref{Proposition_Gikhman} and Lemma \ref{Lemma_estimate_local_time}  we obtain the following modification of Khas'minskii's Lemma (see \cite{Khasminskii59} or \cite{Sznitman98}, Ch.1 Lemma 2.1).
\begin{lem}\label{Lemma_estimation_expectation}
 Let $t_0>0$ be such that $\frac{2\|D\|_\infty\rho(t_0,\lambda)}{B_{\sigma}}<1$ and \eqref{eq_times_inequality_2} hold. Then for all $t\leq t_0$,
 \begin{equation}\label{eq_estimate_expect_exponent}
   \sup_{x\in\R}\mathds{E}e^{2\|D\|_\infty L_t^0(\vp^d(x))}\leq \frac{1}{1-\frac{2\|D\|_\infty}{B_{\sigma}}\rho(t_0, \lambda)}.
 \end{equation}
 \end{lem}
 \vskip 20 pt

Using the inequality \eqref{eq_estimate_expect_exponent} we can estimate the right-hand side of \eqref{eq_Y_moment} for small $t$.

Consider now an arbitrary $t>0$. Put $n=\left[\frac{t}{t_0}\right]+1$, and
$$
s_0=0, s_1=t_0, \dots, s_k=kt_0, \dots, s_{n-1}=(n-1)t_0, s_n=t.
$$
We have
\begin{multline*}
\mathds{E}e^{2\|D\|_\infty{L}_t^0(\vp^d(x))}=\mathds{E}\prod_{k=0}^{n-1}e^{2\|D\|_\infty\left({L}_{s_{k+1}}^0(\vp^d(x))-{L}_{s_k}^0(\vp^d(x))\right)}=\\
\mathds{E}\left(\mathds{E}\left[\prod_{k=0}^{n-1}e^{2\|D\|_\infty\left({L}_{s_{k+1}}^0(\vp^d(x))-{L}_{s_k}^0(\vp^d(x))\right)}\big|\cal{F}_{s_{n-1}}\right]\right)=\\
\mathds{E}\left\{\prod_{k=0}^{n-2}e^{2\|D\|_\infty\left({L}_{s_{k+1}}^0(\vp^d(x))-{L}_{s_k}^0(\vp^d(x))\right)}\right.
\mathds{E}\left.\left(e^{2\|D\|_\infty\left({L}_{s_{n}}^0(\vp^d(x))-{L}_{s_{n-1}}^0(\vp^d(x))\right)}\big|\cal{F}_{s_{n-1}}\right)\right\}.
\end{multline*}
It is not hard to see that
\begin{equation}\label{eq_homogeneous_additive}
P\{{L}_{t+s}^S(\vp(x))={L}_{s}^S(\vp(x))+\theta_s{L}_{t}^S(\vp(x)), \ s\geq0, t\geq0\}=1,
\end{equation}
where $\theta$ is the shift operator.

Using \eqref{eq_homogeneous_additive} and Lemma \ref{Lemma_estimation_expectation} we get for $k=1,\dots,n$,
\begin{multline*}
\mathds{E}\left(e^{2\|D\|_\infty\left({L}_{s_{k}}^0(\vp^d(x))-{L}_{s_{k-1}}^0(\vp^d(x))\right)}\big|\cal{F}_{s_{k-1}}\right)=
\mathds{E}\left(e^{2\|D\|_\infty\theta_{s_{k-1}}L_{s_{k}-s_{k-1}}^0(\vp^d(x))}\big|\cal{F}_{s_{k-1}}\right)\leq \\ \sup_{z\in\R}\mathds{E}e^{2\|D\|_\infty L_{s_{k}-s_{k-1}}^0(\vp^d(z))}\leq\sup_{z\in\R}\mathds{E}e^{2\|D\|_\infty L_{t_{0}}^0(\vp^d(z))}\leq \frac{1}{1-\frac{2\|D\|_\infty}{B_{\sigma}}\rho(t_0, \lambda)} \ \mbox{a.s.}
\end{multline*}
Then

\begin{multline}\label{eq_local_time_moment}
\mathds{E}e^{2\|D\|_\infty{L}_t^0(\vp^d(x))}\leq \frac{1}{1-\frac{2\|D\|_\infty}{B_{\sigma}}\rho(t_0,\lambda)}\mathds{E}\prod_{k=0}^{n-2}e^{2\|D\|_\infty\left({L}_{s_{k+1}}^0(\vp^d(x))-{L}_{s_k}^0(\vp^d(x))\right)}\leq\dots\leq\\
\frac{1}{\left(1-\frac{2\|D\|_\infty}{B_{\sigma}}\rho(t_0,\lambda)\right)^n}= e^{-n\ln\left(1-\frac{2\|D\|_\infty}{B_{\sigma}}\rho(t_0,\lambda)\right)}=e^{-\left(\left[\frac{t}{t_0}\right]+1\right)\ln\left(1-\frac{2\|D\|_\infty}{B_{\sigma}}\rho(t_0,\lambda)\right)}.
\end{multline}
The right-hand side of \eqref{eq_local_time_moment} does not depend on $x$. So we have
\begin{equation}
\sup_{x\in\R}\mathds{E}e^{2\|D\|_\infty L_t^0(\vp^d(x))}\leq e^{-\left(\left[\frac{t}{t_0}\right]+1\right)\ln\left(1-\frac{2\|D\|_\infty}{B_{\sigma}}\rho(t_0,\lambda)\right)}.
\end{equation}
Substituting this inequality into \eqref{eq_Y_moment} we get the following inequality for any $t>0:$
\begin{multline}\label{eq_sup_Y}
\sup_{x\in\R}\mathds{E}|Y_t(x)|\leq
d^{1/2} e^{\left(-\lambda+K_{\alpha}+\frac12K_{\sigma}^2\right)t}\sup_{x\in\R}
\left(\mathds{E}e^{2\|D\|_\infty L_t^0(\vp^d(x))}\right)^{1/2}\leq \\
d^{1/2}e^{\left(-\lambda+K_{\alpha}+\frac12K_{\sigma}^2\right)t-\frac{1}{2}
\left(\frac{t}{t_0}+1\right)\ln\left(1-\frac{2\|D\|_\infty}{B_{\sigma}}\rho(t_0,\lambda)\right)}
=\\
d^{1/2}\frac{1}{\sqrt{1-\frac{2\|D\|_\infty}{B_{\sigma}}\rho(t_0,\lambda)}}
e^{\left(-\lambda+K_{\alpha}+\frac12K_{\sigma}^2
-\frac{1}{2t_0} \ln\left(1-\frac{2\|D\|_\infty}{B_{\sigma}}\rho(t_0,\lambda)\right)\right)t}.
\end{multline}
First, assume that $\Lambda\geq 1/2$. Then for all $\lambda>\Lambda$ and $t\geq 0$,
$$
\rho(t,\lambda)\leq\left(1+\frac{2}{3}\lambda t\right)\sqrt{\left(||\alpha^d||_{\infty}^2+||\sigma||_{\infty}^2\right)t}+
||\alpha^d||_{\infty}t.
$$
From \eqref{eq_sup_Y} we obtain
\begin{equation}\label{eq_estimate_Y}
\sup_{x\in\R}\mathds{E}|Y_t(x)|\leq C_1(\lambda,\alpha,\sigma)
e^{\left(-\lambda+K-\frac{1}{2t_0}\ln\left(1-K_1t_0^{1/2}-K_2t_0-K_3\lambda t_0^{3/2}\right)\right)t},
\end{equation}
where
\begin{equation}\label{eq_C_1}
C_1(\lambda,\alpha,\sigma)=d^{1/2}\frac{1}{\sqrt{1-\frac{2\|D\|_\infty}{B_{\sigma}}\rho(t_0,\lambda)}},
\end{equation}
$$
K=K_{\alpha}+\frac12K_{\sigma}^2,
$$
$$
K_1=\frac{2\|D\|_\infty}{B_{\sigma}}\sqrt{\left(||\alpha^d||_{\infty}^2+||\sigma||_{\infty}^2\right)};
\ \
K_2=\frac{2\|D\|_\infty}{B_{\sigma}}||\alpha^d||_{\infty}; \ \
 K_3=\frac{4\|D\|_\infty}{3B_{\sigma}}\sqrt{\left(||\alpha^d||_{\infty}^2+||\sigma||_{\infty}^2\right)}.
$$

If we show that
$$
\exists \Lambda\geq1/2\ \forall \lambda>\Lambda\ \exists t_0=t_0(\lambda)>0: 0\ \ \
2t_0(-\lambda+K)-\ln(1-K_1\lambda t_0^{3/2}-K_2t_0-K_3t_0^{1/2})<0,
$$
then \eqref{eq_inequality_difference} will follow from \eqref{eq_ineq_difference_of_solutions} and \eqref{eq_estimate_Y}.

Note that for all $\lambda\geq \frac43 K$,
\begin{equation}\label{eq_for_lambda_1}
2t_0(-\lambda+K)\leq -\frac{\lambda t_0}{2}.
\end{equation}
Further, it is easy to see that there exist $\delta>0$ such that
\begin{equation}\label{eq_for_lambda_2}
-\frac12-\ln\left(1-K_2\delta-(K_1+K_3){\delta}^{1/2}\right)<0.
\end{equation}
Put $\Lambda=\max\left\{\frac12, \frac43K,\frac{1}{\delta}\right\}$. Now for each $\lambda>\Lambda$ we can choose $t_0=t_0(\lambda)>0$, which satisfies conditions of Lemma \ref{Lemma_estimation_expectation} and such that $t_0<\frac{1}{\lambda}$. In particular, this implies that $t_0<\delta$. Then using \eqref{eq_for_lambda_1}, \eqref{eq_for_lambda_2} we get
\begin{multline*}
2t_0(-\lambda+K)-\ln(1-K_1\lambda t_0^{3/2}-K_2t_0-K_3t_0^{1/2})\leq\\
-\frac{\lambda t_0}{2}-\ln(1-K_1(\lambda t_0)t_0^{1/2}-K_2 t_0-K_3 t_0^{1/2})\leq\\
-\frac12-\ln(1-K_1t_0^{1/2}-K_2 t_0-K_3 t_0^{1/2})\leq\\
-\frac12-\ln\left(1-K_2\delta-(K_1+K_3){\delta}^{1/2}\right)<0.
\end{multline*}

Hence, if $\lambda>\Lambda$, $t_0$ satisfies the conditions of Lemma \ref{Lemma_estimation_expectation}, and $t_0<\frac1\lambda$, then  there exist $C_1=C_1(\lambda,\alpha,\sigma)>0$ defined by \eqref{eq_C_1} and
$$
C_2=C_2(\lambda,\alpha,\sigma)=2t_0(\lambda-K)+\ln(1-K_1\lambda t_0^{3/2}-K_2t_0-K_3t_0^{1/2})>0
$$
such that the inequality
\[
\sup_{x\in\R}\mathds{E}|Y_t(x)|\leq C_1 e^{-C_2t}
\]
holds.

Similarly we can get the estimate
\begin{equation}\label{eq_Y_to_0}
\sup_{x\in\R}\mathds{E}|Y_t(x)|^p\leq C_1(p) e^{-C_2(p)t}
\end{equation}
for any $p\geq 1$.

Substituting \eqref{eq_Y_to_0} into  \eqref{eq_ineq_difference_of_solutions} we get
\eqref{eq_inequality_difference}.
\end{proof}

\section {Stationary solution}
Let $(\tilde{w}_{1}(t),\dots,\tilde{w}_{m}(t))_{t\geq0}$ and
$(\hat{w}_{1}(t),\dots,\hat{w}_{m}(t))_{t\geq0}$ be standard independent $m$-dimensional Wiener processes.
For $1\leq k\leq m$  define two-sided Brownian motions:
\begin{equation*}
  w_k(t)=\left\{
\begin{aligned}
&\tilde{w}_k(t), t\geq 0,\\
&\hat{w}_k(-t), t<0.
\end{aligned}\right.
\end{equation*}
Let $\cF_t$ be the augmentation of $\sigma$-algebra generated by $\{w_k(s), \ s\leq t,\
1\leq k\leq m\}.$

Consider a $d$-dimensional SDE
\begin{equation}\label{eq_main_stationary}
d\vp_t=\left(-\lambda\vp_t+\alpha(\vp_t)\right)dt+\sum_{k=1}^{m}\sigma_{k}(\vp_t)dw_{k}(t), \ t\in\mathds{R},
\end{equation}
where  $\lambda, \alpha$, $\sigma$ satisfy the conditions of Theorem \ref{Theorem_convergence}.

\begin{dfn}
  We say that $\cF_t$-adapted continuous process $(\vp_t)_{t\in\mathds{R}}$ is a stationary solution to equation \eqref{eq_main_stationary} if for all $s, t\in\mathds{R}$ such that $s\leq t$,
  $$
  \vp_t=\vp_s+\int_s^t\left(-\lambda\vp_u+\alpha(\vp_u)\right)du+\sum_{k=1}^{m}\int_s^t\sigma_{k}(\vp_u)dw_{k}(u) \ \mbox{a.s.},
  $$
  and the process $(\vp_t)_{t\in\mathds{R}}$ is strictly stationary.
\end{dfn}
\begin{thm}
Let $\lambda, \alpha, \sigma$ satisfy the conditions of Theorem \ref{Theorem_convergence}. Then there exists a unique stationary solution to equation \eqref{eq_main_stationary}.
\end{thm}
\begin{proof}  The proof of this theorem is quite standard, see \cite{DaPrato+96}. So we outline only the main steps
without technical details.
{\it Existence.}
Denote by $\vp_{s,t}(x)$, $t\in[s,\infty)$, a solution to the SDE
\begin{equation}\label{eq_main_stationary_1}
\left\{
\begin{aligned}
d\vp_{s,t}(x)&=\left[-\lambda\vp_{s,t}(x)+\alpha(\vp_{s,t}(x))\right]dt+\sum_{k=1}^{m}\sigma_{k}(\vp_{s,t}(x))dw_{k}(t), \ t\geq s,\\
\vp_{s,s}(x)&=x.\\
\end{aligned}\right.
\end{equation}
A stationary solution is looked as a limit in $L_2$ of $\vf_{s,t}(0)$ as $s\to-\infty.$
\begin{lem} \label{Lemma_phi_second_moment}
For all $s\in\mathds{R}$ and $x\in\R$,
\begin{equation}\label{eq_phi_second_moment}
\sup_{t\in[s,\infty)} \mathds{E}|\vp_{s,t}(x)|^2<\infty.
\end{equation}
\end{lem}
\begin{proof}
Let $f\in C^2(\R)$ and $A$ be the infinitesimal generator of the process $(\vp_{s,t}(x))_{t\geq0}$:
$$
Af(x)=\sum_{i=1}^d(-\lambda x^i+\alpha^i(x))\frac{\partial f}{\partial x^i}(x)+\sum_{i,j=1}^d(\sigma(x)\sigma(x)^T)_{i,j}\frac{\partial^2 f}{\partial x^i \partial x^j}(x).
$$
 It is well known (e.g. \cite{Kulik17}, \S 3.2) that to prove \eqref{eq_phi_second_moment} it is enough to verify that there exist $K_1, K_2>0$ such that for all $x\in\R$,
\begin{equation}\label{eq_infinit_K}
A|x|^2\leq K_1-K_2|x|^2.
\end{equation}
We have
$$
A|x|^2=-2\lambda|x|^2+2(\alpha(x),x)+|\sigma|^2.
$$

It is easy to see that \eqref{eq_infinit_K} is satisfied with, for example, $K_1=\frac{\|\alpha\|_\infty^2}{\lambda}+\|\sigma\|_\infty^2$, $K_2=\lambda$. Recall that $\|\alpha\|_\infty=\esssup_{x\in\R}|\alpha(x)|$.
\end{proof}
Let   $t\in\mathds{R}$, $s\leq t$. It
 follows from the uniqueness of the strong solution to \eqref{eq_main_stationary_1}
 that
$$
\vp_{s-p,t}(0)=\vp_{s,t}(\vp_{s-p,s}(0)) \ \mbox{a.s.}
$$
By Theorem \ref{Theorem_convergence} and Lemma \ref{Lemma_phi_second_moment},
\begin{multline*}
\sup_{p\geq 0}\mathds{E}|\vp_{s-p,t}(0)-\vp_{s,t}(0)|=
\sup_{p\geq 0}\mathds{E}|\vp_{s,t}(\vp_{s-p,s}(0))-\vp_{s,t}(0)|\leq\\
\sup_{p\geq 0} C_1 e^{C_2(s-t)}\mathds{E}|\vp_{s-p,s}(0)-0|\leq C_3e^{C_2(s-t)}\to 0, \ s\to-\infty.
\end{multline*}
Here $C_1, C_2$ are constants from Theorem \ref{Theorem_convergence}, $C_3$ is some positive constant that comes from Lemma \ref{Lemma_phi_second_moment}.

Therefore there exists
a limit
\[
\psi(t):=L_2 \lim_{s\to-\infty}\vp_{s,t}(0).
\]
Stationarity of $\psi(t)$ follows from the construction.

 Theorem \ref{Theorem_convergence} and the construction of $\psi(t)$ yield
 that for any $s\leq t:$
$$
 \psi(t)=\vp_{s,t}(\psi(s)) \mbox{ a.s.}
 $$
It follows easily from the last equation that $\psi(t)$
has a continuous modification.

{\it Uniqueness.} Let $(\tilde{\psi}(t))_{t\in\mathds{R}}$ be another stationary
solution, possibly without finite moments.
We have for any $s\leq t:$
\begin{multline}\label{eq_uniqueness_distribution}
\mathds{E}\left(|\tilde{\psi}(t)-\psi(t)|\wedge 1\right)=
\mathds{E}\left(|\vp_{s,t}(\tilde{\psi}(s))-\vp_{s,t}(\psi(s))|\wedge 1\right)\leq\\
\mathds{E}\left(\mathds{E}\left(|\vp_{s,t}(\tilde{\psi}(s))-\vp_{s,t}(\psi(s))|\wedge 1\right)\big|\cal{F}_s\right)
= \mathds{E}\left(\mathds{E}\left(|\vp_{s,t}(x)-\vp_{s,t}(y)|\wedge 1\right)\Big|_{\substack{x=\tilde\psi(s),
 \\ y=\psi(s)}}\right)\leq
\\
\mathds{E}\left(\left(\mathds{E} |\vp_{s,t}(x)-\vp_{s,t}(y)|\right)\Big|_{\substack{x=\tilde\psi(s),
 \\ y=\psi(s)}}\wedge 1
\right)\leq
\mathds{E}\left(|x-y|C_1e^{-C_2(t-s)} \Big|_{\substack{x=\tilde\psi(s),
 \\ y=\psi(s)}} \wedge 1
 \right)=
\\
\mathds{E}\left(|\tilde{\psi}(s)-\psi(s)|C_1e^{-C_2(t-s)}\wedge 1\right)\leq\\
 \mathds{E}\left[|\psi(s)|C_1e^{-C_2(t-s)}\wedge 1\right]+\mathds{E}\left[
|\tilde{\psi}(s)|C_1e^{-C_2(t-s)}\wedge 1\right]=\\
\mathds{E}\left[|\psi(0)|C_1e^{-C_2(t-s)}\wedge 1\right]+\mathds{E}\left[
|\tilde{\psi}(0)|C_1e^{-C_2(t-s)}\wedge 1\right].
\end{multline}
Here $C_1, C_2$ are constants from \eqref{eq_inequality_difference}.
 By Lebesgue's dominated convergence theorem, the right-hand side of
 \eqref{eq_uniqueness_distribution} tends to zero as $s\to-\infty$.
Hence
  $\mathds{E}|\tilde\psi(t)-\psi(t)|\wedge 1= 0$. This and continuity of
  $(\tilde\psi(t)), (\psi(t))$ yield
  \[
  \P(\tilde\psi(t)=\psi(t),\ t\in {\mathds{R}})=1.
  \]
\end{proof}

\section{Appendix}
\begin{proof}[Proof of Lemma \ref{Lemma_moment_of_Y}]
Put
$$
\tau^N=\inf\{t\geq0:t+\int_0^t|Y_s(x)|^2ds+L^0_t(\vp^d(x))\geq N\}.
$$

By It\^{o}'s formula,
\begin{multline*}
  e^{h(t\wedge \tau^N)}|Y_{t\wedge \tau^N}(x)|^2=|Y_0(x)|^2+\int_0^{t\wedge \tau^N} e^{h(s)}|Y_s(x)|^2(2\lambda-2K_{\alpha}-K_{\sigma}^2)ds-\\
  2\int_0^{t\wedge \tau^N} e^{h(s)}|Y_s(x)|^2|D(\vp_s(x))|dL_s^0(\vp(x))+
  2\int_0^{t\wedge \tau^N} e^{h(s)}\sum_{i,j=1}^d Y_s^{ij}(x)dY_s^{ij}(x)ds+ \\
  \sum_{k=1}^m\sum_{i,j=1}^d \int_0^ {t\wedge \tau^N} e^{h(s)}\left(\sum_{r=1}^{d}\nabla\sigma_k^{ir}(\vp_s(x))Y_s^{rj}(x)\right)^2
  ds=\\
  |Y_0(x)|^2+\int_0^{t\wedge \tau^N} e^{h(s)}|Y_s(x)|^2\big(2\lambda-2K_{\alpha}-K_{\sigma}^2)ds-\\2\int_0^{t\wedge \tau^N} e^{h(s)}|Y_s(x)|^2|D(\vp_s(x))|dL_s^0(\vp(x))- \\
  2\lambda \int_0^{t\wedge \tau^N} e^{h(s)}\sum_{i,j=1}^d(Y_s^{ij}(x))^2ds+2\int_0^{t\wedge \tau^N} e^{h(s)}\sum_{i,j,q=1}^d Y_s^{ij}(x)\nabla \alpha^{iq}(\vp_s(x))Y_s^{qj}(x)ds+\\
  2\int_0^{t\wedge \tau^N} e^{h(s)}\sum_{i,j,q=1}^d Y_s^{ij}(x)D^{iq}(\vp_s(x))Y_s^{qj}(x)dL_s^0(\vp^d(x))+\\
  2\int_0^{t\wedge \tau^N} e^{h(s)}\sum_{k=1}^m\sum_{i,j,q=1}^d Y_s^{ij}(x)\nabla\sigma_k^{iq}(\vp_s(x))Y_s^{qj}(x)dw_k(s)+\\
  \sum_{k=1}^m\sum_{i,j=1}^d \int_0^{t\wedge \tau^N} e^{h(s)}\left(\sum_{r=1}^{d}\nabla\sigma_k^{ir}(\vp_s(x))Y_s^{rj}(x)\right)^2ds.
\end{multline*}

Using the Cauchy-Schwarz inequality we get that for each $t>0$,
\begin{equation*}
\left|\sum_{i,j,q=1}^d Y_t^{ij}(x)\nabla\alpha^{iq}(\vp_t(x))Y_t^{qj}(x)\right|\leq
|Y_t(x)|^2|\nabla\alpha(\vp_t(x))|.
\end{equation*}
Then taking into account Remark \ref{Remark_zero_time} we obtain
\begin{equation}\label{eq_a1}
\int_0^{t\wedge\tau^N}\left|\sum_{i,j,q=1}^d Y_s^{ij}(x)\nabla\alpha^{iq}(\vp_s(x))Y_s^{qj}(x)\right|ds\leq K_{\alpha}\int_0^{t\wedge\tau^N}|Y_s(x)|^2ds,
\end{equation}
where $K_\alpha$ is defined by \eqref{eq_K_alpha_def}.

Similarly,
\begin{multline}\label{eq_a2}
\int_0^{t\wedge \tau^N} e^{h(s)}\left|\sum_{i,j,q=1}^d Y_s^{ij}(x)D^{iq}(\vp_s(x))Y_s^{qj}(x)\right|dL_s^0(\vp^d(x))\leq \\ \|D\|_\infty\int_0^{t\wedge \tau^N} e^{h(s)}|Y_s(x)|^2dL_s^0(\vp^d(x)).
\end{multline}
Further, for $t>0$,
\begin{multline}\label{eq_a3}
\sum_{k=1}^m\sum_{i,j=1}^d \left(\sum_{r=1}^{d}\nabla\sigma_k^{ir}(\vp_t(x))Y_t^{rj}(x)\right)^2\leq\\
\sum_{k=1}^m\sum_{i,j=1}^d \left(\sum_{r=1}^{d}\big(\nabla\sigma_k^{ir}(\vp_t(x))\big)^2\right)\left(\sum_{r=1}^{d}\big(Y_t^{rj}(x)\big)^2\right)\leq |\nabla\sigma(x)|^2|Y_t(x)|^2\leq K_\sigma^2|Y_t(x)|^2,
\end{multline}
where $K_\sigma$ is defined by \eqref{eq_K_sigma_def}.

Taking into account \eqref{eq_a1}--\eqref{eq_a3} we get
\begin{multline*}
e^{h(t\wedge \tau^N)}|Y_{t\wedge \tau^N}(x)|^2=|Y_0(x)|^2+(2\lambda-2K_{\alpha}-K_{\sigma}^2)\int_0^{t\wedge \tau^N} e^{h(s)}|Y_s(x)|^2ds-\\
2\int_0^{t\wedge \tau^N} e^{h(s)}|Y_s(x)|^2\|D\|_\infty dL_s^0(\vp(x))-
2\lambda \int_0^{t\wedge \tau^N} e^{h(s)}|Y_s(x)|^2ds+\\2dK_{\alpha}\int_0^{t\wedge \tau^N} e^{h(s)}|Y_s(x)|^2ds+
2\int_0^{t\wedge\tau^N} e^{h(s)}\|D\|_\infty|Y_s(x)|^2dL_s^0(\vp^d(x))+\\
K_{\sigma}^2\int_0^te^{h(s)}|Y_s(x)|^2ds+\\
2\int_0^{t\wedge\tau^N}e^{h(s)}\sum_{k=1}^m\sum_{i,j,q=1}^d Y_s^{ij}(x)\nabla\sigma_k^{iq}(\vp_s(x))Y_s^{qj}(x)dw_k(s)\leq
|Y_0(x)|^2+M(t\wedge\tau^N),
\end{multline*}
where
$$
M(t\wedge\tau^N)=2\int_0^{t\wedge\tau^N}e^{h(s)}\sum_{k=1}^m\sum_{i,j,q=1}^d Y_s^{ij}(x)\nabla\sigma_k^{iq}(\vp_s(x))Y_s^{qj}(x)dw_k(s), t\geq0,
$$
is a square integrable martingale. Then for all $t\geq0$,
$$
\mathds{E}e^{h(t\wedge \tau^N)}|Y_{t\wedge \tau^N}(x)|^2\leq |Y_0(x)|^2=|E|^2=d.
$$
Passing to the limit as $N\to\infty$ and applying Fatou's lemma we get that for all $T>0$,
$$
\sup_{t\in[0,T]}\mathds{E}e^{h(t)}|Y_{t}(x)|^2\leq d.
$$
\end{proof}

\begin{proof}[Proof of Lemma \ref{Lemma_estimate_local_time}]
The process $(\vp_t(x))_{t\geq0}$ is a multidimensional  semimartingale.
Set
\begin{equation}\label{eq_Tanaka_formula}
\widetilde{L}_t^0(\vp^d(x))=2(\vp_t^d(x))^+-2(x^d)^+-2\int_0^t\mathds{1}_{\vp_s^d(x)>0}d\vp_s^d(x), \ t\geq0.
\end{equation}
This process is also called a local time of the process $(\vp^d_t(x))_{t\geq0}$ at zero and satisfies the equality (see \cite{Revuz+99}, Ch. VI, Corollary (1.9) and Th. (1.7))
\begin{multline*}
\widetilde{L}_t^0(\vp^d(x))=\lim_{\varepsilon\downarrow 0}\frac{1}{2\varepsilon}\int_0^t \mathds{1}_{(-\varepsilon,\varepsilon)}(\varphi_s^d(x))d\langle\vp^d(x),\vp^d(x)\rangle_s=\\
=\lim_{\varepsilon\downarrow 0}\frac{1}{2\varepsilon}\int_0^t\mathds{1}_{(-\varepsilon,\varepsilon)}(\varphi_s^d(x))\sum_{k=1}^d\big(\sigma_k^d(\vp_s(x))\big)^2ds,
\end{multline*}
which holds almost surely.

Using \eqref{eq_local_time_0} and (B3) in which we put $\theta^*=(0,0,\dots,0,1)$ we get that, almost surely,
$$
\widetilde{L}_t^0(\vp^d(x))\geq B_{\sigma}\lim_{\varepsilon\downarrow 0}\frac{1}{2\varepsilon}\int_0^t \mathds{1}_{(-\varepsilon,\varepsilon)}(\varphi_s^d(x))ds=B_{\sigma}L_t^0(\vp^d(x))=B_{\sigma}L_t^0(\vp(x)).
$$
So
\begin{equation}\label{eq_local_times_inequality_1}
L_t^0(\vp^d(x))\leq \frac{1}{B_\sigma}\widetilde{L}_t^0(\vp^d(x)).
\end{equation}
Consequently, to estimate $\mathds{E}L_t^0(\vp(x))$ it is enough to get an estimation for $\mathds{E}\widetilde{L}_t^0(\vp^d(x))$.

Since the local time $L_t^0(\vp^d(x))=L_t^S(\vp(x))$ does not increase until the first time when the process $(\vp_t(x))_{t\geq0}$ reaches the hyperplane $S$, it follows from the strong Markov property of the process $(\vp_t(x))_{t\geq0}$ that
\begin{equation}\label{eq_inequality_local_time}
\mathds{E}{L}_t^0(\vp^d(x))\leq \sup_{\tilde{x}\in S}\mathds{E}L_{t}^0(\vp^d(\tilde{x})).
\end{equation}


To estimate the expectation of the local time let us estimate each term in the right-hand side of Tanaka's formula \eqref{eq_Tanaka_formula}.
Note that for all $t\geq 0$,
\begin{equation}\label{eq_ineq_for_phi_plus}
\mathds{E}(\vp^d_t(\tilde{x}))^+\leq \mathds{E}|\vp_t^d(\tilde{x})|\leq \sqrt{\mathds{E}\big(\vp_t^d(\tilde{x})\big)^2}.
\end{equation}
By It\^o's formula, for any $\tilde{x}\in S$
\begin{multline}\label{eq_phi_d_2}
(\vp^d_t(\tilde{x}))^2=\vp_0^d(\tilde x))^2+2\int_0^t\vp^d_s(\tilde{x})d\vp_s^d(\tilde{x})+\sum_{k=1}^m\int_0^t\big(\sigma_k^d(\vp_s(\tilde{x}))\big)^2ds=\\
\int_0^t\left[-2\lambda(\vp_s^d(\tilde{x}))^2+2\alpha^d(\vp_s(\tilde{x}))\vp_s^d(\tilde{x})+\sum_{k=1}^m\big(\sigma_k^d(\vp_s(\tilde{x}))\big)^2\right]ds+\\
2\int_0^t\vp_s^d(\tilde{x})\sum_{k=1}^m\sigma_k^d(\vp_s(\tilde{x}))dw_k(s).
\end{multline}
Consider the expression in the square brackets. We have
\begin{equation*}
-2\lambda(\vp_s^d(\tilde{x}))^2+2\alpha^d(\vp_s(\tilde{x}))\vp_s^d(\tilde{x})+\sum_{k=1}^m\big(\sigma_k^d(\vp_s(\tilde{x}))\big)^2\leq f(\vp_s^d(\tilde x)),
\end{equation*}
where $f(x)=-2\lambda x^2+2||\alpha^d||_{\infty}^2x+||\sigma||_{\infty}^2$. The function $f$ attains the global maximum at
$x_{max}=\frac{||\alpha^d||_{\infty}^2}{2\lambda}$, and $f(x_{max})=\frac{||\alpha^d||_{\infty}^2}{2\lambda}+||\sigma||_{\infty}^2$. Substituting the maximum value of $f$ into \eqref{eq_phi_d_2} we get
$$
(\vp^d_t(\tilde{x}))^2\leq \left(\frac{||\alpha^d||_{\infty}^2}{2\lambda}+||\sigma||_{\infty}^2\right)t+2\sum_{k=1}^m\int_0^t\vp_s^d(\tilde{x})\sigma_k^d(\vp_s(\tilde{x}))dw_k(s).
$$
Here we use the fact that $\vp^d_0(\tilde{x})=0$. Note that $\sum_{k=1}^m\int_0^t\vp_s^d(\tilde{x})\sigma_k^d(\vp_s(\tilde{x}))dw_k(s)$ is a  local square integrable martingale.

Using localization and Fatou's lemma we obtain
$$
\mathds{E}(\vp^d_{t}(\tilde{x}))^2\leq \left(\frac{||\alpha^d||_{\infty}^2}{2\lambda}+||\sigma||_{\infty}^2\right)t.
$$
Hence, by \eqref{eq_ineq_for_phi_plus}
\begin{equation}\label{eq_phi_plus}
\mathds{E}(\vp^d_t(\tilde{x}))^+\leq \sqrt{\left(\frac{||\alpha^d||_{\infty}^2}{2\lambda}+||\sigma||_{\infty}^2\right)t}.
\end{equation}
Further,
\begin{multline*}
-\int_0^t\mathds{1}_{\vp^d_s(\tilde{x})>0}d\vp_s^d(\tilde{x})=\\ \lambda\int_0^t\mathds{1}_{\vp^d_s(\tilde{x})>0}\vp_s^d(\tilde{x})ds-
\int_0^t\mathds{1}_{\vp^d_s(\tilde{x})>0}\alpha^d(\vp_s(\tilde{x}))ds-\sum_{k=1}^m\int_0^t\mathds{1}_{\vp^d_s(\tilde{x})>0}\sigma_k^d(\vp_s(\tilde{x}))dw_k(s).
\end{multline*}
Using \eqref{eq_phi_plus} we get
\begin{multline}\label{eq_indicator_phi}
-\mathds{E}\int_0^t\mathds{1}_{\vp^d_s(\tilde{x})>0}d\vp_s^d(\tilde{x})\leq \|\alpha^d\|_{\infty}t+\lambda\mathds{E}\int_0^t(\vp_s^d(\tilde{x}))^+ds\leq\\
\|\alpha^d\|_{\infty}t+\frac{2\lambda}{3}\sqrt{\left(\frac{||\alpha^d||_{\infty}^2}{2\lambda}+\|\sigma\|_{\infty}^2\right)t^3}.
\end{multline}
Taking into account \eqref{eq_phi_plus}, \eqref{eq_indicator_phi} we obtain
\begin{equation*}
\mathds{E}\widetilde{L}_t^0(\vp^d(\tilde{x}))\leq  \rho(t,\lambda),
\end{equation*}
where
$$
\rho(t,\lambda)=||\alpha^d||_{\infty}t+\left(1+\frac{2}{3}\lambda t\right)\sqrt{\left(\frac{||\alpha^d||_{\infty}^2}{2\lambda}+\|\sigma\|_{\infty}^2\right)t}.
$$
Then  \eqref{eq_inequality_local_time} implies that for each $x\in\R$,
\begin{equation*}
\mathds{E}\widetilde{L}_t^0(\vp^d(x))\leq  \rho(t,\lambda),
\end{equation*}
and
\begin{equation*}
\sup_{x\in\R}\mathds{E}\widetilde{L}_t^0(\vp^d(x))\leq  \rho(t,\lambda).
\end{equation*}
By \eqref{eq_local_times_inequality_1},
$$
\sup_{x\in\R}\mathds{E}L_t^0(\vp^d(x))\leq \frac{\rho(t,\lambda)}{B_\sigma}.
$$
\end{proof}

\section*{Acknowledgment}
The authors appreciate to the reviewer for scrutinize reading and oversights found.


\end{document}